\documentclass[reqno, 12pt,draft]{amsart}
\usepackage{amsmath,amsthm,amsfonts,amssymb}
\usepackage[cp1251]{inputenc}
   \usepackage[english]{babel}
\date{}
\newtheorem{theorem}{Theorem}[section]

\newtheorem{corollary}[theorem]{Corollary}

\numberwithin{equation}{section}

\begin{document}

\centerline{\sc The Dirichlet problem for double divergence form elliptic equations}

\centerline{\sc  with measures as boundary conditions}

\vspace*{0.2cm}

\centerline{\sc V.I. Bogachev$^{a,b,c}$, S.V. Shaposhnikov$^{a,b}$}

\vskip .1in

$^{a}$ Department of Mechanics and Mathematics, Moscow State University, 119991 Moscow, Russia

$^{b}$ National Research University Higher School of Economics, Myasnitskaya 20, 101000 Moscow, Russia

$^{c}$ The corresponding author, email: vibogach@mail.ru

\vskip .1in

\vspace*{0.2cm}

{\small Abstract.
We introduce and study the  Dirichlet problem for double divergence  form elliptic  equations
with coefficients of low regularity and boundary conditions given by general Borel measures.
Under broad assumptions we establish the solvability of this problem. It is also
shown that a solution to a double divergence form equation
on a domain serves as a solution to the  Dirichlet problem on inner subdomains.
The obtained results are applied to the study of properties of solutions to stationary
Fokker--Planck--Kolmogorov equations.
}



\section{Introduction}

In  this paper we introduce and study the  Dirichlet problem for double divergence  form elliptic  equations with coefficients of low regularity,
an important particular case of which are stationary   Fokker--Planck--Kolmogorov equations,
with boundary conditions given by general Borel measures.
Under broad assumptions we establish the solvability of this problem, moreover, it is shown that a solution to a double divergence form equation
on a domain serves as a solution to the  Dirichlet problem on inner subdomains.
The obtained results are applied to the study of properties of solutions to  Fokker--Planck--Kolmogorov equations and are also used for
correcting some flaws detected in a number of recent papers.

Let $\Omega\subset \mathbb{R}^d$ be a nonempty open set.
Given a mapping $x\mapsto S(x)=(s^{ij}(x))_{1\le i, j\le d}$
from $\Omega$ to the space of symmetric matrices and a  mapping $x\mapsto T(x)=(t^i(x))_{1\le i\le d}$
from $\Omega$ to~$\mathbb{R}^d$, where the functions $s^{ij}, t^i$ belong to the
class $L^1_{loc}(\Omega)$ of functions integrable on compact sets in~$\Omega$,
we set
$$
{\rm div}^2 S=\sum_{i, j=1}^d\partial_{x_i}\partial_{x_j}s^{ij}, \quad {\rm div}\,T=\sum_{i=1}^d\partial_{x_i}t^i,
$$
where the derivatives are understood in the sense of distributions.
The mappings $S$ and $T$ generate the elliptic operator
\begin{multline*}
L_{S, T}u(x)={\rm trace}\bigl(S(x)D^2u(x)\bigr)+\langle T(x), \nabla u(x)\rangle
\\
=\sum_{i,j\le d} s^{ij}(x)\partial_{x_i}\partial_{x_j}u(x)+\sum_{i\le d}t^i(x)\partial_{x_i}u(x).
\end{multline*}
The formally adjoint operator $L_{S, T}^{*}$ has the form
$$
L_{S, T}^{*}u={\rm div}^2 (uS)-{\rm div}(uT).
$$

Let us consider the double divergence form elliptic equation
\begin{equation}\label{eq1}
{\rm div}^2(\varrho A)-{\rm div}(\varrho b)={\rm div}^2G-{\rm div}h,
\end{equation}
where the vector fields $b(x)=(b^i(x))_{1\le i\le d}$ and $h(x)=(h^i(x))_{1\le i\le d}$
and the symmetric matrices $A(x)=(a^{ij}(x))_{1\le i, j\le d}$ and $G(x)=(g^{ij}(x))_{1\le i, j\le d}$
are given by Borel functions $b^i$, $h^i$, $a^{ij}$ and~$g^{ij}$, respectively, on the domain $\Omega$.

A function $\varrho\in L^1_{loc}(\Omega)$ is called a solution to equation (\ref{eq1})
if 
$$
a^{ij}\varrho, \ b^i\varrho, \ h^i, g^{ij}\in L^1_{loc}(\Omega)
$$
 and for every function $u$ from the
 class $C_0^{\infty}(\Omega)$ of infinitely differentiable
functions with compact support in $\Omega$ the equality
$$
\int_{\Omega}L_{A, b}u(x)\varrho(x)\,dx=
\int_{\Omega}L_{G, h}u(x)\,dx
$$
holds.

By means of the operators $L_{A, b}^{*}$ and $L_{G, h}^{*}$ equation (\ref{eq1}) can be written in a shorter form
$$
L_{A, b}^{*}\varrho=L_{G, h}^{*}1.
$$
In the case where $h=0$, $G=0$ this is the stationary  Kolmogorov  (or Fokker--Planck--Kolmogorov) equation
$$
L_{A,b}^*\varrho=0,
$$
which under broad conditions is satisfied for the densities of stationary
measures of the diffusion process with the   diffusion matrix $\sqrt{2}A$ and  the drift~$b$, see~\cite{BKRS}.
The  Kolmogorov  equation is usually considered with respect to Borel measures $\mu$ on $\Omega$ and is also written in the
form
$$
L_{A,b}^*\mu=0
$$
with the same interpretation as the identity
$$
\int_{\Omega}L_{A, b}u(x)\, \mu(dx)=0, \quad u\in C_0^{\infty}(\Omega),
$$
where the local integrability of $a^{ij}, b^i$ with respect to~$\mu$ is required.
However, in the  case of a nondegenerate  matrix $A$ (which is considered below) any solution $\mu$ is given by a density
(see \cite{BKR} and \cite{BKRS}).

We also observe that one can include in the equation a term $c\varrho$ with a function $c$ by taking the
 operator $L_{A,b,c}\varrho=L_{A,b}\varrho +c\varrho$
in place of $L_{A,b}$, but for the   existence theorem in the  Dirichlet problem this leads
to rather cumbersome conditions even for constant~$c$,
while in all other assertions under suitable conditions of local integrability of $c$ the function $c\varrho$ can be written as
 $c\varrho={\rm div} h_0$ with some vector field~$h_0$ or as
$c\varrho={\rm div}^2 G_0$ with some matrix mapping~$G_0$, which leads merely to replacing of~$h$ or~$G$.

We assume throughout that the coefficients of the equation satisfy the following conditions:

\vspace*{0.2cm}

{\bf (H1)} all functions $a^{ij}$ are continuous on $\Omega$ and there
is a nondecreasing continuous function $\omega$ on $[0, +\infty)$ such that $\omega(0)=0$ and
the inequality $\|A(x)-A(y)\|\le\omega(|x-y|)$ holds for all $x, y\in\Omega$,

\vspace*{0.2cm}

{\bf (H2)} there exists a  number $\theta>0$ such that
$\theta {\rm I}\le A(x)\le \theta^{-1}{\rm I}$ for all $x\in\Omega$, where ${\rm I}$ is the unit matrix,

\vspace*{0.2cm}

{\bf (H3)} for some $p>d$ we have the inclusions
$$
b^i\in L^p_{loc}(\Omega), \quad h^i\in L^1_{loc}(\Omega), \quad g^{ij}\in L^{p'}_{loc}(\Omega),
$$
where $p'=p/(p-1)$ and the class $L^p_{loc}(\Omega)$ consists of measurable functions integrable on compacta in $\Omega$ to power~$p$.

\vspace*{0.2cm}

Note that unlike the matrix  $A$ the matrix $G$ need not be nonnegative definite, i.e., the operator $L_{G, h}$
need not be elliptic.

Apart solutions to equation (\ref{eq1}) on $\Omega$ we consider solutions to the  Dirichlet problem for equation (\ref{eq1})
on domains contained with closure in $\Omega$.

Let $ D $ be a bounded domain with smooth boundary  (of class $C^2$) and $\overline{D}\subset\Omega$.
By $W^{p,2}(D)$ we denote the Sobolev  space of functions in $L^p(D)$ for which all
first and second order generalized partial derivatives  also belong to~$L^p(D)$.

Let $\eta$ be a bounded Borel measure on $\partial  D $ (possibly, signed).
By $\sigma_{d-1}$ we denote the standard surface measure on $\partial  D $ (the Hausdorff measure) and by $\nu(x)$
the unit outer normal to $\partial D $ at the point~$x$.

We say that a function $\varrho\in L^{p'}( D )$ is a solution to the  Dirichlet problem
\begin{equation}\label{dir}
\left\{\begin{array}{lc}
  {\rm div}^2 (\varrho A)-{\rm div}(\varrho b)={\rm div}^2G-{\rm div}\,h, & x\in  D , \\
  \\
  \varrho|_{\partial  D }=\eta+\kappa\sigma_{d-1},
\end{array}\right.
\end{equation}
where the function $\kappa$ is defined by the formula
\begin{equation}\label{dens}
\kappa(x)=\frac{\langle G(x)\nu(x), \nu(x)\rangle}{\langle A(x)\nu(x), \nu(x)\rangle},
\end{equation}
if for every function $u\in W^{p, 2}( D )\cap C^1(\overline{ D })$ vanishing on $\partial  D $ the equality
\begin{equation}\label{int}
\int_{ D }L_{A, b}u(x)\varrho(x)\,dx
=\int_{ D }L_{G, h}u(x)dx+\int_{\partial D }\langle A(x)\nabla u(x), \nu(x)\rangle\,\eta(dx)
\end{equation}
holds.

It should be emphasized that the expression $\kappa\sigma_{d-1}$ in \eqref{dir} in general is just a symbol and does not enter
the defining identity~\eqref{int}, but it becomes
meaningful if $G$ is continuous: in that case $\kappa$ is also continuous and can serve as a density with respect to the measure~$\sigma_{d-1}$.
We shall see below that in case of smooth coefficients the
second line in \eqref{dir} holds as the equality for measures $\varrho\, \sigma_{d-1}= (\varrho -\kappa)\, \sigma_{d-1}+\kappa\, \sigma_{d-1}$ ,
i.e., $\eta=(\varrho -\kappa)\, \sigma_{d-1}$.

Let us explain whence the boundary condition in \eqref{dir} appears in the case of $A={\rm I}$, $h=0$, $G=0$, and smooth~$b$.
Let $\varrho$ be a solution to the equation $L_{{\rm I}, b}^*\varrho=0$ in~$\Omega$.
Then $\varrho\in C^\infty(\Omega)$, hence for all functions $u\in C^2(\overline{ D })$ vanishing on~$\partial D$ we have by the
integration by parts formula
\begin{multline*}
\int_{ D } L_{{\rm I},b}u\varrho\, dx=
\int_{\partial  D }\langle \nabla u,\nu\rangle \varrho\, d\sigma_{d-1}-
\int_{ D } \langle \nabla u,\nabla \varrho\rangle  \, dx +
\int_{ D } \langle \nabla u, b\rangle \varrho  \, dx
\\
=\int_{\partial  D }\langle \nabla u,\nu\rangle \varrho\, d\sigma_{d-1},
\end{multline*}
since
$$
\int_{ D } \langle \nabla u,\nabla \varrho\rangle  \, dx -
\int_{ D } \langle \nabla u, b\rangle \varrho  \, dx
=-\int_{ D }  u[\Delta\varrho -{\rm div}(\varrho b)]\, dx=0
$$
due to the equation $\Delta\varrho -{\rm div}(\varrho b)=0$. Thus, here the measure $\eta$ is given by density $\varrho$ with respect to
the surface measure, i.e., is the restriction of the solution itself. The situation is similar for smooth nonzero~$G$. In that case
we have an extra term on the right, which is the integral of $L_{G,0}u$ over~$D$, which by the integration by parts formula equals
\begin{multline*}
-\sum_{i,j} \int_D \partial_{x_i}g^{ij}\partial_{x_j} u\, dx+\int_{\partial D} \langle G\nabla u, \nu\rangle \, d\sigma_{d-1}
\\
=\sum_{i,j} \int_D \partial_{x_j}\partial_{x_i}g^{ij}u\, dx +\int_{\partial D} \langle G\nabla u, \nu\rangle \, d\sigma_{d-1}.
\end{multline*}
Since $\nabla u(x)= \langle\nabla u(x), \nu(x)\rangle \nu(x)$ on $\partial D$, which follows
from the equality $u=0$ on~$\partial D$, we have
$$
\langle \nabla u,\nu\rangle \varrho =\langle\nabla u, \nu\rangle \varrho, \quad
\langle G\nabla u, \nu\rangle=\langle G\nu, \nu\rangle\langle\nabla u, \nu\rangle
\quad \hbox{on } \partial D.
$$
So again \eqref{int} holds and we can take $\eta=(\varrho -\langle G\nu, \nu\rangle)\, \sigma_{d-1}$ in \eqref{dir}.
Similar calculations for general smooth $A$ lead to $\eta=(\varrho -\kappa)\, \sigma_{d-1}$.

In the general case considered below, solutions to the double divergence form  elliptic
equation need not belong to the Sobolev class, which makes a notable distinction of such equations from direct
 elliptic equations and  divergence  form equations. For example, every bounded measurable function $\varrho$ on the real line such that
 the  function $1/\varrho$ is locally   bounded  satisfies
 the double divergence form  elliptic
equation  with $A(x)=1/\varrho(x)$ and $b=h=0$, $G=0$ (which becomes the Kolmogorov  equation).
 The continuity of $A$ ensures only the continuity of~$\varrho$,
but not more. In the multidimensional case the known example of Bauman \cite{Bau} shows that
the  continuity of $A$ does not imply even the local boundedness of~$\varrho$.
Note that if a solution is not continuous or Sobolev, then there is no natural way to restrict it to surfaces.
For surveys of the theory of  Fokker--Planck--Kolmogorov equations see \cite{BKRS}, \cite{BRS24}, and \cite{BSumn}.
In the recent years, there is a growth of interest in such equations with   diffusion matrices of low
regularity, see, for example, the last two surveys and
\cite{Lee1}, \cite{Lee2}, \cite{Lee3}, \cite{LeeST}, \cite{LeeT}, where one can find additional references, in particular,
related to applications to diffusion processes.
The  Dirichlet problem  for direct and divergence form elliptic  equations is well studied, see, for example, the books \cite{GT},
\cite{Krylov96} and the recent papers \cite{Krylov08}, \cite{Gu19}, \cite{Gu24}.

The main  results of this paper are as follows.

1) Under conditions {\rm (H1)}, {\rm (H2)}, {\rm (H3)} and the
additional assumption of the continuity of $G$ every nonnegative solution to equation
(\ref{eq1}) on $\Omega$ is a solution to the  Dirichlet problem on every bounded domain  $ D $ with smooth boundary contained
with closure in $\Omega$ with the boundary condition given by some measure $\eta$ on $\partial D $,
and this is also true for locally bounded signed solutions. Moreover, this result extends to the
nonnegative solutions to the equation only on~$D$ (not on all of $\Omega$) belonging to $L^{p'}(D)$.

2) The Dirichlet problem (\ref{dir}) has a unique solution, and for this solution  some a priori estimates are obtained.

3) For sufficiently ``nice'' smooth approximations of the functions $a^{ij}$, $g^{ij}$, $b^i$, $h^i$ and the measure $\eta$,
 smooth solutions to the  corresponding Dirichlet problems   converge to the  solution to the  Dirichlet  problem for the equation with
the original coefficients and the boundary condition with the original measure $\eta$.

Note that it was assertion 1)  that motivated us to study the   Dirichlet problem for double divergence form
equations with measures as boundary conditions. This problem was previously investigated in the
case of boundary conditions from $L^r(\partial D )$ for some $r\ge 1$ or a boundary  condition given by a continuous function.
The Dirichlet problem  for double divergence form elliptic equations is concerned in the papers
\cite{Sjorg1}, \cite{Sjorg2}, \cite{Esc}, \cite{DEK18}, \cite{DKK25}, and~\cite{K25}.
In  \cite{Sjorg1}, \cite{Sjorg2}, and~\cite{DKK25}, the Dirichlet  problem  for
weak solutions of class $C(\overline{ D })$ with boundary conditions given by continuous functions
on $\partial D $ was studied.  Weak solutions to the  Dirichlet  problem with boundary conditions
given by functions from $L^{q}(\partial D )$ for some $q\ge 1$ were considered in \cite{Esc}.
A~survey of recent results in this area is given in \cite{K25}. As far as we know, the Dirichlet  problem
 for double divergence form elliptic equations with boundary conditions given by
 measures on $\partial D $ has not yet been studied. Note that in the cited papers the matrix $A$
was assumed to be H\"older continuous or satisfying Dini's condition (classical or  in mean).
The Dirichlet problem  for divergence form linear and quasilinear elliptic equations with boundary conditions given by measures is
well studied  (see, for example, \cite{Veron2004} and \cite{MarcVeron13}).

In the study of the  Dirichlet  problem we are  particularly interested in the question
of smooth approximations of solutions, which plays a crucial role in the proof of the continuity of solutions to double divergence form elliptic equations.
In the special case $b=h=0$, it is stated  in \cite[Theorem 1.10]{DK17} that if the matrices $A$ and $G$ satisfy the Dini
mean oscillation condition  (see the definition in front of Theorem~\ref{th4}),
then every solution $\varrho$ to equation (\ref{eq1}), belonging to $L^2_{loc}(\Omega)$, has a continuous version.
The same assertion with the class $L^1_{loc}(\Omega)$ in place of $L^2_{loc}(\Omega)$ is formulated in Theorem~2.15
in the recent paper~\cite{K25}. However, in the proof of Theorem~1.10 in \cite{DK17} the
assumption about the local boundedness of the solution is substantially used, i.e., membership in the class $L^\infty_{loc}(\Omega)$
 (see the reasoning on p.~429 in~\cite{DK17}).
 On p.~422 in~\cite{DK17} a reference is given to the approximation procedure from \cite{D12}  for divergence form equations,
 but no details are given for the case of double divergence form equations considered in the cited theorem.
 Professor Hongjie Dong has kindly communicated these details.
  In~\cite{K25}, after Theorem~2.15,
 a~remark is made explaining that the condition of local boundedness can be omitted if one approximates the solution by smooth solutions
to analogous equations with smooth coefficients and this remark is justified only in the case, where
one considers not merely solutions to equation (\ref{eq1}), but solutions to the Dirichlet problem, which by virtue of uniqueness of solutions
enables one to assert that the approximations converge exactly to the original solution and not to some other solution.
However, it is not clear at all that an arbitrary solution $\varrho$ to equation (\ref{eq1})
on any open ball $B$ belonging with closure to the domain~$\Omega$ is in some sense
a solution to the  Dirichlet problem. Note that the justification of
Harnack's inequality for nonnegative solutions to equation (\ref{eq1}) with $b=0$, $G=0$ and $h=0$
given in~\cite{DEK18} also employs the cited Theorem~1.10 from~\cite{DK17}. In turn, generalizations of Harnack's inequality
to double divergence form elliptic equations of a  more general type obtained in
\cite{BRS23} and \cite{GKim} use Harnack's inequality from~\cite{DEK18}.
For convenience of further references we explain how to reduce the general case to the case of smooth solutions
of equations with smooth coefficients.
Thus, the aforementioned statements  about existence of continuous versions of any solutions and the validity of Harnack's inequality
 for nonnegative solutions are valid in their original formulations.

\section{Main results}

Let $B(a,R)$ denote the open ball of radius $R$ centered at $a$.

\begin{theorem}\label{th1}
Suppose that in addition to conditions {\rm (H1)}, {\rm (H2)}, {\rm (H3)}
the functions $g^{ij}$ are continuous on $\Omega$ and
a function $\varrho\in L^{p'}_{loc}(\Omega)$
is a solution to equation {\rm (\ref{eq1})}
on $\Omega$ such that $\varrho(x)\ge 0$ for almost all $x\in\Omega$.
Let $ D $ be a bounded domain with smooth boundary $\partial D $ of class $C^2$ and $\overline{ D }\subset\Omega$.
Then there exists a bounded Borel measure $\eta$ on the surface $\partial D $ such that the function $\varrho$ is a solution
to the  Dirichlet problem
{\rm (\ref{dir})} on $ D $ with boundary condition $\eta+\kappa\sigma_{d-1}$ on $\partial D $. If it is known in advance  that
$\varrho\in L^{\infty}_{loc}(\Omega)$, then the assumption of nonnegativity of $\varrho$ can be omitted, moreover,
the corresponding measure $\eta$ has a bounded density with respect to the surface measure $\sigma_{d-1}$ on $\partial D $.
\end{theorem}
\begin{proof}
Let $\delta>0$. By $ D _{\delta}$ we denote the $\delta$-neighborhood of the set $ D $.
For sufficiently small $\delta$ the inclusion $ D _{4\delta}\subset\Omega$ holds.
Let $\psi\in C_0^{\infty}(\mathbb{R}^d)$ be such that $\psi\ge 0$, $\|\psi\|_{L^1(\mathbb{R}^d)}=1$
and $\psi$ vanishes outside the ball $B(0, 1)$. If $0<\varepsilon<\delta$, we set
$$
\psi_{\varepsilon}(x)=\varepsilon^{-d}\psi(\varepsilon^{-1}x).
$$
 For an integrable function $f$ we denote by $f_{\varepsilon}$ the  convolution $f*\psi_{\varepsilon}$:
$$
f_{\varepsilon}(x)=\int_{\mathbb{R}^d} f(x-y)\psi_{\varepsilon}(y)\, dy
=\int_{\mathbb{R}^d} f(y)\psi_{\varepsilon}(x-y)\, dy.
$$
Extend the functions $a^{ij}, g^{ij}, b^i, h^i$ by zero outside of~$\Omega$.
Let $x\in  D _{2\delta}$. Substituting in the integral equality defining the solution
the function $u(y)=\psi_{\varepsilon}(x-y)$, we obtain the equality
$$
\sum_{i, j=1}^d\partial_{x_i}\partial_{x_j}(a^{ij}\varrho)_{\varepsilon}-
\sum_{i=1}^d\partial_{x_i}(b^{i}\varrho)_{\varepsilon}
=
\sum_{i, j=1}^d\partial_{x_i}\partial_{x_j}(g^{ij})_{\varepsilon}-
\sum_{i=1}^d\partial_{x_i}(h^{i})_{\varepsilon}.
$$
Multiplying this equality by a function $u\in W^{2, p}( D )\cap C^1(\overline{ D })$
vanishing on $\partial D $, with the  aid of integration by parts we obtain
\begin{multline}\label{apr}
\int_{ D }\Bigl[{\rm trace}\bigl((A\varrho)_{\varepsilon}(x)D^2u(x)\bigr)+\langle (b\varrho)_{\varepsilon}(x), \nabla u(x)\rangle\Bigr]\,dx
\\
=\int_{ D }\Bigl[{\rm trace}\bigl(G_{\varepsilon}(x)D^2u(x)\bigr)+\langle h_{\varepsilon}(x), \nabla u(x)\rangle\Bigr]\,dx
\\
+\int_{\partial D }\langle(A\varrho)_{\varepsilon}(x)\nabla u(x), \nu(x)\rangle\,\sigma_{d-1}(dx)-
\int_{\partial D }\langle G_{\varepsilon}(x)\nabla u(x), \nu(x)\rangle\,\sigma_{d-1}(dx),
\end{multline}
where $\sigma_{d-1}$ is the standard surface  measure on $\partial D $.
According to \cite[Lemma 1.3.5]{MarcVeron13}, there exists a function $\varphi\in C^2(\overline{ D })$ such
that $\varphi(x)=0$ and $\langle\nabla\varphi(x), \nu(x)\rangle=1$ if $x\in\partial D $.
Since the function $\varphi$ is constant on the boundary $\partial D $, its gradient there is proportional to the normal.
Therefore, the equality $\langle\nabla\varphi(x), \nu(x)\rangle=1$
for $x\in\partial D $ implies the equality $\nabla\varphi(x)=\nu(x)$.
Substituting into equality (\ref{apr}) the function $\varphi$ in place of~$u$, we obtain
\begin{multline*}
\int_{\partial D }\langle(A\varrho)_{\varepsilon}(x)\nu(x), \nu(x)\rangle\,\sigma_{d-1}(dx)
\\
=
\int_{ D }\Bigl[{\rm trace}\bigl((A\varrho)_{\varepsilon}(x)D^2\varphi(x)\bigr)
+\langle(b\varrho_{\varepsilon})(x), \nabla\varphi(x)\rangle\Bigr]\,dx
\\
-\int_{ D }\bigl[{\rm trace}\bigl(G_{\varepsilon}(x)D^2\varphi(x)\bigr)+\langle h_{\varepsilon}(x), \nabla\varphi(x)\rangle\bigr]\,dx
+\int_{\partial D }\langle G_{\varepsilon}(x)\nu(x), \nu(x)\rangle\,\sigma_{d-1}(dx).
\end{multline*}
Taking into account that $\varrho$ is nonnegative,  we have
$$
\langle(A\varrho)_{\varepsilon}(x)\nu(x), \nu(x)\rangle=
\int_{\mathbb{R}^d}  \langle A(y)\nu(x), \nu(x)\rangle\varrho(y)\psi_{\varepsilon}(x-y)\,dy\ge
\theta\varrho_{\varepsilon}(x).
$$
Therefore, the estimate
$$
\int_{\partial D }\varrho_{\varepsilon}(x)\,\sigma_{d-1}(dx)\le C
$$
holds, where the constant
\begin{multline*}
C=\theta^{-1}\|\varphi\|_{C^2(\overline{ D })}\int_{ D _{3\delta}}
\Bigl[d\|A(x)\|+|b(x)|\Bigr]\varrho(x)\,dx
\\
+\theta^{-1}\|\varphi\|_{C^2(\overline{ D })}\int_{ D _{3\delta}}
\Bigl[d\|G(x)\|+|h(x)|\Bigr]\,dx
+\theta^{-1}\sigma_{d-1}(\partial D )\sup_{ D _{3\delta}}\|G(x)\|
\end{multline*}
does not depend on $\varepsilon$. By virtue of this estimate and compactness of $\partial D $, the hypothesis of Prohorov's theorem
is fulfilled (see \cite[Theorem~8.6.2]{B07}), therefore, there exists a sequence $\varepsilon_j\to 0$ such that the
sequence of nonnegative measures $\varrho_{\varepsilon_j}\,\sigma_{d-1}$
converges weakly  to some nonnegative measure $\widetilde{\eta}$ on $\partial D $. Observe that
$$
\|(A\varrho)_{\varepsilon}(x)-A(x)\varrho_{\varepsilon}(x)\|
\le \omega(\varepsilon)\varrho_{\varepsilon}(x).
$$
Therefore, for every function $u\in W^{p,2}( D )\cap C^1(\overline{ D })$ vanishing on $\partial D $ the expression
$$
\biggl|\int_{\partial D } \big\langle\bigl((A\varrho)_{\varepsilon}(x)-A(x)\varrho_{\varepsilon}\bigr)
\nabla u(x), \nu(x)\big\rangle\,\sigma_{d-1}(dx)\biggr|
$$
is estimated from above by the quantity
$$
C\omega(\varepsilon)\max_{\partial D }|\nabla u|
$$
and tends to zero as $\varepsilon\to 0$.
Moreover, the following  equality is true:
$$
\lim_{j\to\infty}\int_{\partial D }\langle A(x)\nabla u(x), \nu(x)\rangle\varrho_{\varepsilon_j}(x)\, \sigma_{d-1}(dx)=
\int_{\partial D }\langle A(x)\nabla u(x), \nu(x)\rangle\,\widetilde{\eta}(dx).
$$
Observe that as $\varepsilon\to 0$ one has convergence
$(\varrho A)_{\varepsilon}\to \varrho A$ in $L^{p'}(D)$ and $(b\varrho)_{\varepsilon}\to b\varrho$ in $L^1(D)$.
Replacing $\varepsilon$ by $\varepsilon_j$ in (\ref{apr}) and taking into account the  remarks made above,
we pass to the limit as $j\to\infty$ in (\ref{apr}) and obtain the equality
\begin{multline*}
\int_{ D }L_{A, b}u(x)\varrho(x)\,dx
-\int_{ D }L_{G, h}u(x)\,dx
\\
=\int_{\partial D }\langle A(x)\nabla u(x), \nu(x)\rangle\,\widetilde{\eta}(dx)-
\int_{\partial D }\langle G(x)\nabla u(x), \nu(x)\rangle\,\sigma_{d-1}(dx),
\end{multline*}
Since $u=0$ on $\partial D $, we have $\nabla u(x)=\langle\nabla u(x), \nu(x)\rangle\nu(x)$ if $x\in\partial D $.
Hence
\begin{multline*}
\int_{\partial D }\langle A(x)\nabla u(x), \nu(x)\rangle\,\widetilde{\eta}(dx)-
\int_{\partial D }\langle G(x)\nabla u(x), \nu(x)\rangle\,\sigma_{d-1}(dx)
\\
=
\int_{\partial D }\langle A(x)\nabla u(x), \nu(x)\rangle\,\eta(dx),
\end{multline*}
where $\eta=\widetilde{\eta}-\kappa\sigma_{d-1}$ and the density $\kappa$ is given by formula (\ref{dens}).
Thus, we arrive at the equality
$$
\int_{ D }L_{A, b}u(x)\varrho(x)\,dx
-\int_{ D }L_{G, h}u(x)\,dx
=\int_{\partial D }\langle A(x)\nabla u(x), \nu(x)\rangle\,\eta(dx).
$$
Suppose  now that the function $\varrho$ is locally bounded and can be sign variable.
Then for all $x\in\partial D $ the estimate
$|\varrho_{\varepsilon}(x)|\le \|\varrho\|_{L^{\infty}( D _{3\delta})}$ holds.
Therefore, one can pick a sequence $\varepsilon_j\to 0$ such that
the sequence of functions $\varrho_{\varepsilon_j}$ converges weakly in $L^2(\partial D )$, where  $\partial D $
is equipped with the surface measure, to some bounded Borel function on $\partial D $. The rest of the justification
does not differ from the proof in the case of a nonnegative solution~$\varrho$.
\end{proof}

The assumption $\varrho\ge 0$ was used to obtain the uniform boundedness of measures
$\varrho_\varepsilon\, \sigma_{d-1}$ on~$\partial D$.

We do not know whether the  measure $\eta$ constructed in Theorem \ref{th1} is always
absolutely continuous with respect to the surface measure $\sigma_{d-1}$ on $\partial D $.
However, in a somewhat more special situation one can show that in a certain sense
almost all such measures are absolutely continuous. Suppose that a ball $B(x_0, R_1)$ is contained with closure in $\Omega$
and a function $\varrho\in L^{p'}_{loc}(\Omega)$ is a nonnegative solution to equation \eqref{eq1} on~$\Omega$.
Let us fix a Borel version of the function $\varrho$ on~$\Omega$.

\begin{theorem}
For almost all $R\in(0, R_1)$ the function $\varrho$ is a solution to the  Dirichlet  problem \eqref{dir}
on $B(x_0, R)$ with boundary condition $\eta+\kappa\sigma_{d-1}$, which has
density $\varrho$ with respect to  the standard surface measure $\sigma_{d-1}$ on $\partial B(x_0, R)$.
\end{theorem}
\begin{proof}
As in Theorem \ref{th1}, we denote by $\varrho_{\varepsilon}$ the convolution $\varrho*\psi_{\varepsilon}$.
Let $0<R<R_1$. It is seen from the proof of Theorem \ref{th1} that
the nonnegative measures $\varrho_{\varepsilon}\,\sigma_{d-1}$
on $\partial B(x_0, R)$ are  bounded in variation uniformly
with respect to $\varepsilon$ and $R\in(0, R_1)$, so there exists
 a sequence $\{\varepsilon_j\}$ converging to zero such that the measures $\varrho_{\varepsilon_j}\,\sigma_{d-1}$
 converge weakly to some nonnegative measure $\widetilde{\eta}$. Moreover,
the function $\varrho$ is a solution to the  Dirichlet problem \eqref{dir}
with boundary condition $\eta+\kappa\sigma_{d-1}$, where $\eta=\widetilde{\eta}-\kappa\sigma_{d-1}$.
We prove that the measure $\widetilde{\eta}$ does not depend on our choice of the sequence $\{\varepsilon_j\}$.
Suppose that for some other sequence $\{\varepsilon_j'\}$ we obtained a measure $\widetilde{\xi}$.
Set $\xi=\widetilde{\xi}-\kappa\sigma_{d-1}$.
Since the function $\varrho$ is a solution to the  Dirichlet  problem with boundary condition $\xi+\kappa\sigma_{d-1}$,
for every function $u\in W^{p, 2}(B(x_0, R))\cap C^1(\overline{B}(x_0, R))$ vanishing on $\partial B(x_0, R)$ we have the equality
$$
\int_{\partial B(x_0, R)}\langle A(x)\nabla u(x), \nu(x)\rangle\, (\eta-\xi)(dx)=0.
$$
Let $h$ be an arbitrary smooth function on $\mathbb{R}^d$.
According to \cite[Lemma 1.3.5]{MarcVeron13},
there exists a function $\varphi\in C^2(\overline{B}(x_0, R))$ such
that $\varphi(x)=0$ and $\langle\nabla\varphi(x), \nu(x)\rangle=h$
if $x\in\partial B(x_0, R)$.
Since the function $\varphi$ is constant on the boundary $\partial B(x_0, R)$, we have
$\nabla\varphi(x)=h(x)\nu(x)$ for $x\in\partial B(x_0, R)$. Therefore,
$$
\int_{\partial B(x_0, R)}h(x)\langle A(x)\nu(x), \nu(x)\rangle\, (\eta-\xi)(dx)=0.
$$
Using that the function $h$ is arbitrary and the matrix $A$ is nondegenerate, we conclude that $\eta=\xi$.
Thus, the measure $\widetilde{\eta}$ does not depend on our choice of the sequence $\{\varepsilon_j\}$.
Therefore, the measures $\varrho_{\varepsilon}\sigma_{d-1}$ converge weakly to $\widetilde{\eta}$ as $\varepsilon\to 0$.
For every $R\in (0,R_1)$ we denote by $\widetilde{\eta}_R$ the limiting measure for $\varrho_{\varepsilon}\sigma_{d-1}$
on $\partial B(x_0, R)$. Let $f\in C(\overline{B}(x_0, R_1))$ and $0<R<R_1$.
Observe that
$$
\lim_{\varepsilon\to 0}\int_{B(x_0, R)}f(x)\varrho_{\varepsilon}(x)\,dx=\int_{B(x_0, R)}f(x)\varrho(x)\,dx.
$$
Since
$$
\int_{B(x_0, R)}f(x)\varrho_{\varepsilon}(x)\,dx=
\int_0^R \int_{\partial B(x_0, r)}f(x)\varrho_{\varepsilon}(x)\, \sigma_{d-1}(dx)\,dr,
$$
the following  equality is valid:
$$
\lim_{\varepsilon\to 0}\int_{B(x_0, R)}f(x)\varrho_{\varepsilon}(x)\,dx=
\int_0^{R}\int_{\partial B(x_0, r)}f(x)\, \widetilde{\eta}_r(dx)\,dr.
$$
Therefore, for all $R\in (0,R_1)$ we have
$$
\int_0^R \int_{\partial B(x_0, r)}f(x)\varrho(x)\, \sigma_{d-1}(dx)\,dr
=\int_0^{R} \int_{\partial B(x_0, r)}f(x)\, \widetilde{\eta}_r(dx)\,dr.
$$
This implies that for almost all $R$ in the interval $(0, R_1)$ we have
\begin{equation}\label{remeq}
\int_{\partial B(x_0, R)}f(x)\varrho(x)\, \sigma_{d-1}(dx)=
\int_{\partial B(x_0, R)}f(x)\, \widetilde{\eta}_R(dx).
\end{equation}
Let $\mathcal{F}$ be a  countable collection of continuous functions on $\overline{B}(x_0, R_1)$ such
that for every $R\in (0,R_1)$ the set of restrictions of functions from $\mathcal{F}$ to $\partial B(x_0, R)$
is everywhere dense in $C(\partial B(x_0, R))$. For such a collection one can take
all polynomials on $\mathbb{R}^d$ with rational coefficients. We can assume that
for almost all $R$ from the interval $(0, R_1)$ equality (\ref{remeq})
holds true for all functions $f$ from $\mathcal{F}$. Then for almost all $R\in(0, R_1)$
the equality $\widetilde{\eta}_R=\varrho\, \sigma_{d-1}$ is true.
\end{proof}

Our next theorem is an analog of the Herglotz--Doob theorem for
superharmonic functions (see \cite[Theorem 1.4.1]{MarcVeron13}).
Its distinction from Theorem~\ref{th1} is that the boundary value on $\partial D $ arises for the solutions from $L^{p'}( D )$
to the equation only on~$D$, but not on the whole domain~$\Omega$.

\begin{theorem}\label{th11}
Suppose that in addition to conditions {\rm (H1)}, {\rm (H2)}, {\rm (H3)}
the functions $g^{ij}$ are continuous on the domain $\Omega$, $ D $ is a bounded
domain with smooth boundary  $\partial D $ of class $C^2$ and $\overline{ D }\subset\Omega$.
Assume also that $\varrho\in L^{p'}( D )$ is a solution to equation {\rm (\ref{eq1})}
on $ D $ {\rm (}not on all of $\Omega${\rm )} and $\varrho(x)\ge 0$ for almost all $x\in D $.
Then there exists a bounded Borel measure $\eta$ on~$\partial D $ such that the function $\varrho$ is a solution to the Dirichlet
problem  {\rm (\ref{dir})} on $ D $ with boundary condition $\eta+\kappa\sigma_{d-1}$ on $\partial D $.
If it is given in advance that $\varrho\in L^{\infty}( D )$, then the assumption of nonnegativity of $\varrho$ can be omitted.
\end{theorem}
\begin{proof}
According to \cite[Lemma 1.3.4]{MarcVeron13}, there exists a uniform $C^2$-exhaustion of the
domain  $ D $ by domains $ D_j$ with  boundaries of class $C^2$ and $\overline{ D}_j\subset D $.
This means (see \cite[Definition 1.3.1]{MarcVeron13}) that the domains $ D_j$ increase to~$D$, the closure of $D_j$ is contained in $D_{j+1}$,
every point $a\in\partial D $ possesses a neighborhood $U(a)$ and a coordinate system  $(y_1, \ldots, y_d)$
such that the set $U(a)\cap\partial D_j$ is the graph of a function $f_j(y_1, \ldots, y_{d-1})$ of class $C^2$,
the set $U(a)\cap\partial D$ is the graph of a function $f(y_1, \ldots, y_{d-1})$ of class $C^2$
and the functions $f_j$ converge to $f$ in $C^2(U(a)\cap(\mathbb{R}^{d-1}\times\{0\}))$.
In the main model  case where $D$ is a ball the domains $D_j$ are also balls with the same center.

By Theorem \ref{th1} the solution $\varrho$ on every domain  $ D _j$ satisfies the  Dirichlet  problem \eqref{dir}
with boundary condition $\eta_j+\kappa\sigma_{d-1}$ on $\partial D _j$,
where $\eta_j$ is a finite nonnegative Borel measure on $\partial D_j$.
Further we consider $\eta_j$ as a measure on $\overline{D}$ with support in $\partial D _j$.
According to  \cite[Lemma 1.3.5]{MarcVeron13}, there exists a sequence of
functions $\varphi_j\in C^2(\overline{ D _j})$ such
that $\varphi_j(x)=0$ and $\langle\nabla\varphi_j(x), \nu_j(x)\rangle=1$ whenever $x\in\partial D _j$,
where $\nu_j(x)$ is the outer normal to $\partial D_j$ at the point~$x$,
and $\sup_j\|\varphi_j\|_{C^2(\overline{ D_j})}<\infty$. As in the proof of the previous theorem,
we observe that $\nabla\varphi_j(x)=\nu_j(x)$ if $x\in\partial D_j$. Substituting the function $\varphi_j$
in the integral equality (\ref{int}), we obtain
$$
\int_{ D _j}L_{A, b}\varphi_j(x)\varrho(x)\,dx
=\int_{ D _j}L_{G, h}\varphi_j(x)dx+\int_{\partial D _j}\langle A(x)\nu_j(x), \nu_j(x)\rangle\,\eta_j(dx).
$$
Therefore,
$$
\eta_j(\partial D _j)\le \theta^{-1}\sup_j\|\varphi_j\|_{C^2(\overline{D_j})}\biggl(
\int_{ D }\bigl[d\|A(x)\|+|b(x)|\bigr]\varrho(x)\,dx+
\int_{ D }\bigl[d\|G(x)\|+|h(x)|\bigr]\,dx\biggr).
$$
Since the set $\overline{D}$ is compact, passing to a subsequence we can
assume that the measures $\eta_j$ converge weakly to some bounded nonnegative measure $\eta$ on $\overline{ D }$.
Using that every point $a$ in the support of the measure $\eta$ is the limit
of a convergent sequence of points $a_j$ from the supports of the measures $\eta_j$, we conclude that $a\in\partial D $.
Thus, the support of the measure $\eta$ is contained in $\partial D $.
Let us verify that the function $\varrho$ is a solution to the  Dirichlet  problem with boundary condition
$\eta+\kappa\sigma_{d-1}$ on $\partial D$.
Observe that it suffices to establish the validity of the integral identity (\ref{int})
for functions $u\in C^2(\overline{ D })$ vanishing on $\partial{D}$.
Moreover, we can  assume that $u\in C^2(\mathbb{R}^d)$. There is a  sequence of
functions $u_j\in C^2(\mathbb{R}^d)$ such that $\{u_j\}$ converges to $u$ in $C^2(\overline{D})$
and $u_j=0$ on $\partial D_j$. With the aid of partition of unity we reduce the construction of $u_j$ to the case, where
either 1)~the support of the function $u$ is contained in the interior of all $D_j$
or 2)~the support of $u$ is contained in the aforementioned neighborhood $U(a)$
of some point $a$ from the boundary of~$D$. In the first case we set $u_j=u$ and in the second case we set
$$
u_j(y_1, \ldots, y_d)=u\bigl(y_1, \ldots, y_{d-1}, y_d-f_j(y_1, \ldots, y_{d-1})+f(y_1, \ldots, y_{d-1})\bigr).
$$
Since $\varrho$ is a solution to the  Dirichlet problem on the domain $D_j$, we have
$$
\int_{ D_j}L_{A, b}u_j(x)\varrho(x)\,dx
=\int_{ D_j}L_{G, h}u_j(x)dx+\int_{\partial D_j}\langle A(x)\nabla u_j(x), \nu_j(x)\rangle\,\eta_j(dx).
$$
Using  convergence of $\{u_j\}$ to $u$ in $C^2(\overline{ D })$, we obtain the equalities
$$
\lim_{j\to\infty}\int_{D_j}L_{A, b}u_j(x)\varrho(x)\,dx
=\int_{D}L_{A, b}u(x)\varrho(x)\,dx,
$$
$$
\lim_{j\to\infty}\int_{ D _j}L_{G, h}u_j(x)\, dx=\int_{ D }L_{G, h}u(x)\, dx.
$$
Let us justify the equality
$$
\lim_{j\to\infty}\int_{\partial D _j}\langle A(x)\nabla u_j(x), \nu_j(x)\rangle\,\eta_j(dx)=
\int_{\partial D}\langle A(x)\nabla u(x), \nu(x)\rangle\,\eta(dx).
$$
Since $\{\nabla u_j\}$  converges to $\nabla u$ uniformly on $\overline{ D }$,  it suffices to establish the equality
$$
\lim_{j\to\infty}\int_{\partial D_j}\langle A(x)\nabla u(x), \nu_j(x)\rangle\,\eta_j(dx)=
\int_{\partial D}\langle A(x)\nabla u(x), \nu(x)\rangle\,\eta(dx).
$$
Due to  \cite[Lemma 1.3.4]{MarcVeron13}, we can assume that
for sufficiently large numbers $j$ the set $ D _j$ has the form $\{x\in D \colon d(x)>1/j\}$, where
$d(x)={\rm dist}(x, \partial D )$. According to \cite[Proposition 1.3.2]{MarcVeron13}, for some $\delta>0$ the
function $d$ is twice continuously differentiable on the set
$$
F_{\delta}=\{x\in\overline{ D }\colon \, d(x)\le\delta\},
$$
for every point $x\in F_{\delta}$ there exists a unique  point $s(x)\in\partial D$ such that
$$
|x-s(x)|=d(x), \quad x=s(x)-d(x)\nu(x),
$$
the function $s$ is continuously differentiable on $F_{\delta}$ and
$\nabla d(x)$ tends to $-\nu(s(x))$ when $x$ tends to $s(x)$. It is clear that $s(x)=x$ if $x\in\partial D $.
In addition,
$$
\nu_j(x)=-\nabla d(x)/|\nabla d(x)| \quad \hbox{if } \ x\in\partial D_j.
$$
Therefore, $\sup_{x\in\partial D_j}|\nu_j(x)-\nu(s(x))|$ tends to zero as $j\to\infty$.
Thus, it suffices to justify the equality
$$
\lim_{j\to\infty}\int_{\partial D_j}\big\langle A(x)\nabla u(x), \nu(s(x))\big\rangle\,\eta_j(dx)=
\int_{\partial D}\langle A(x)\nabla u(x), \nu(s(x))\rangle\,\eta(dx),
$$
but it  is a corollary of  weak convergence of the measures $\eta_j$ to the measure $\eta$.
Thus, passing to the limit as $j\to\infty$, we obtain the equality
$$
\int_{D} L_{A, b}u(x)\varrho(x)\,dx
=\int_{D} L_{G, h}u(x)dx+\int_{\partial D}\langle A(x)\nabla u(x), \nu(x)\rangle\,\eta(dx).
$$
Therefore, the function $\varrho$ is a solution to the  Dirichlet  problem on $ D $ with boundary condition
$\eta+\kappa\sigma_{d-1}$ on $\partial D$.

In the case where $\varrho$ is a signed solution of class $L^{\infty}(D)$, by Theorem \ref{th1}
every measure $\eta_j$ on $D_j$ is given by a bounded density with respect to  the standard
surface  measure $\sigma_{d-1}$ and the $L^{\infty}$-norm of this density is estimated from above by $\|\varrho\|_{L^{\infty}(D)}$.
Therefore, the total variations of the measures $\eta_j$ (as measures on $\overline{D}$) are uniformly bounded, hence
we can pick a weakly convergent subsequence. Then we can repeat the reasoning given above.
\end{proof}

Suppose that, as above, $D$ is a bounded domain with smooth boundary,
 $\overline{D}\subset\Omega$ and we are given a  bounded Borel measure $\eta$ on $\partial  D$ (possibly, signed).
Let $\|\eta\|_{TV}$ denote the total variation norm of the measure~$\nu$.

 For every function $f\in L^p_{loc}(\Omega)$ and every bounded set $E$ contained with closure in~$\Omega$,
by the absolute continuity of the Lebesgue integral  there exists a nondecreasing nonnegative continuous function $\omega_{f, E}$ on
 $[0, +\infty)$ for which $\omega_{f, E}(0)=0$ and for every $\delta>0$ and every Lebesgue measurable set $Q\subset E$
with   Lebesgue measure less than $\delta$ one has
$$
\int_Q|f(x)|^p\,dx\le\omega_{f, E}(\delta).
$$
A analogous function $\omega_{v, E}(\delta)$ can be  also found in the case of a vector field $v$ with $|v|\in L^{p}_{loc}(\Omega)$.

The existence of solutions to the  Dirichlet problem is justified in Theorem \ref{th3} below.
We now address the uniqueness problem.

\begin{theorem}\label{th2}
A solution $\varrho$ to the  Dirichlet problem {\rm (\ref{dir})} is unique and satisfies the estimate
$$
\|\varrho\|_{L^{p'}( D )}\le C\Bigl(\|\eta\|_{TV}+\|h\|_{L^1( D )}+\|G\|_{L^{p'}( D )}\Bigr),
$$
where the constant $C$ depends only on $\theta$, $\omega$, $\|b\|_{L^p( D )}$, $\omega_{b, D}$ and the domain  $ D $.
\end{theorem}
\begin{proof}
Uniqueness  follows from the stated a priori estimate and the linearity of the problem.
Let $f\in C_0^{\infty}( D )$. According to \cite[Theorem 2.4]{VIT92} (see also \cite{VIT93}),
there exists a solution $u\in W^{p, 2}( D )\cap W^{p, 1}_0( D )$ to the  Dirichlet problem
$$
\left\{\begin{array}{lc}
   L_{A, b}u=f, & x\in  D , \\
   u|_{\partial D }=0.\\
\end{array}\right.
$$
Moreover, one has the estimate
$$
\|u\|_{W^{p, 2}( D )}\le C_1\|f\|_{L^p( D )},
$$
where the constant $C_1$ depends only on $\theta$, $\omega$, $\|b\|_{L^p( D )}$, $\omega_{b, D}$ and~$ D $.
Note that in the case of a bounded coefficient $b$ an analogous result is covered by \cite[Theorem 9.15]{GT},
and the case of more singular coefficients is studied in \cite{Krylov23}.

Since $p>d$, by the Sobolev  embedding theorem the function $u$ has a version continuously differentiable on $\overline{ D }$,
vanishing on $\partial D $ and satisfying the estimate
$$
\|u\|_{C^1(\overline{ D })}\le C_2\|u\|_{W^{p, 2}( D )},
$$
where $C_2$ depends only on $p$ and $ D $.

Substituting the function $u$ into the integral identity (\ref{int}), we obtain
$$
\int_{ D } f(x)\varrho(x)\,dx=\int_{\partial D }\langle A(x)\nabla u(x), \nu(x)\rangle\,\eta(dx)
+\int_{ D } L_{G, h}u(x)\,dx.
$$
Applying the  estimates of the norms $\|u\|_{C^1(\overline{ D })}$ and $\|u\|_{W^{p, 2}( D )}$ stated above,
we arrive at the  inequality
$$
\int_{ D } f(x)\varrho(x)\,dx\le C_3\Bigl(\|\eta\|_{TV}+\|h\|_{L^1(B)}+\|G\|_{L^{p'}( D )}\Bigr)
\|f\|_{L^p( D )},
$$
where the constant $C_3$ depends only on $\theta$, $\omega$, $\|b\|_{L^p( D )}$, $\omega_{b, D}$ and $ D $.
Due to the Riesz theorem on representation of  continuous linear functionals on $L^p( D )$ we obtain
the desired estimate of the $L^{p'}$-norm of the function~$\varrho$.
\end{proof}

\section{Smooth approximations of solutions}

We call a sequence of infinitely smooth functions
$a^{ij}_n$, $b^i_n$, $g^{ij}_n$, $h^i_n$  on
$\Omega$ an {\it admissible approxima\-tion} of the functions $a^{ij}$, $b^i$, $g^{ij}$, $h^i$
if, for every compact set $K\subset\Omega$, the sequence $\{a^{ij}_n\}$ converges to $a^{ij}$ uniformly on $K$,
the matrices $A_n=(a^{ij}_n)$ are symmetric and satisfy on $K$ conditions {\rm (H1)} and {\rm (H2)} with
a number $\theta_K$ and a function $\omega_K$ independent of $n$, the sequence
$\{b^i_n\}$ converges to $b^i$ in $L^p(K)$, there
exists a continuous nondecreasing nonnegative function $\widetilde{\omega}_{K}$ on $[0, +\infty)$ such
that the function $\widetilde{\omega}_{K}$ can be taken for $\omega_{b_n, K}$ for all~$n$,
the sequence $\{g^{ij}_n\}$ converges to $g^{ij}$
in $L^{p'}(K)$, and the sequence $\{h^i_n\}$ converges to $h^i$ in $L^1(K)$ for all $i,j$.

An admissible approximation can be constructed in the following way. Extend the function $a^{ij}$ by the value $\delta^{ij}$
outside $\Omega$ and all other  functions $b^i$, $g^{ij}$ and $h^i$ by zero outside~$\Omega$.
Let $\psi\in C_0^{\infty}(\mathbb{R}^d)$ be such that $\psi\ge 0$, $\|\psi\|_{L^1(\mathbb{R}^d)}=1$
and the support of $\psi$ is contained in the open ball $B(0, 1)$. Set
$$
\psi_{1/n}(x)=n^{d}\psi(nx),\quad
\Omega_n=\{x\in\Omega\colon {\rm dist}(x, \partial\Omega)>1/n\}
$$
 and denote by $I_{\Omega_n}$ the indicator function of~$\Omega_n$.
 Then the sequence of functions
$$
a^{ij}_n=a^{ij}*\psi_{1/n}, \ b^{i}_n=(b^{i}I_{\Omega_n})*\psi_{1/n}, \ g^{ij}_n=(g^{ij}I_{\Omega_n})*\psi_{1/n},  \
h^i_n=(h^iI_{\Omega_n})*\psi_{1/n}
$$
is an admissible approximation.

Let $\{a^{ij}_n\}$, $\{b^i_n\}$, $\{g^{ij}_n\}$, $\{h^i_n\}$ be an admissible approximation of class $C^\infty$,
$ D $ a bounded domain with smooth boundary,
$\overline{D}\subset\Omega$, and let $\{\eta_n\}$ be a sequence of smooth functions on $\partial  D $.

For every $n$ consider the  Dirichlet problem
\begin{equation}\label{dirapr}
\left\{\begin{array}{lc}
  {\rm div}^2\bigl(v A_n\bigr)-{\rm div}\bigl(vb_n\bigr)={\rm div}^2G_n-{\rm div}\,h_n, & x\in D , \\
  \\
  v|_{\partial D }=\eta_n\sigma_{d-1}+\kappa_n\sigma_{d-1},
\end{array}\right.
\end{equation}
where the function $\kappa_n$ is defined by  formula {\rm (\ref{dens})} with $A_n$ and $G_n$ in place of $A$ and $G$.

It is well known  (see, for example, \cite[Chapter 8]{GT}) that the  Dirichlet problem for the
divergence form elliptic equation
$$
{\rm div}\bigl(A_n\nabla v- v \widetilde{b}_n\bigr)=F_n,
$$
where
$$
\widetilde{b}_n^i=b^i_n-\sum_{j=1}^d\partial_{x_j}a^{ij}_n, \quad F_n={\rm div}^2G_n-{\rm div}\,h_n,
$$
on the ball $ D $ with the infinitely differentiable boundary condition $\eta_n-\kappa_n$ on $\partial D $ has a solution~$\varrho_n$
 infinitely differentiable on $\overline{ D }$.
It is clear that $\varrho_n$  solves the  Dirichlet problem~(\ref{dirapr}).

Now from Theorem \ref{th2} we derive convergence of solutions $\varrho_n$ to the solution $\varrho$
of the  Dirichlet  problem (\ref{dir}).

\begin{theorem}\label{th3}
Suppose that a sequence of measures $\eta_n\,\sigma_{d-1}$, given by smooth
densities $\eta_n$ with respect to the measure $\sigma_{d-1}$,  converges weakly to a measure $\eta$ on $\partial D $.
Then the sequence of functions $\varrho_n$ converges weakly in $L^{p'}( D )$
to the solution $\varrho$ of the  Dirichlet problem {\rm (\ref{dir})},
in particular, for every finite Borel measure $\eta$ on $\partial D $
there exists a solution to the  Dirichlet problem {\rm (\ref{dir})}.
Moreover, if the measures $\eta_n\sigma_{d-1}$ converge to the measure $\eta$ on $\partial D $ in variation,
then the sequence $\{\varrho_n\}$ converges to the solution $\varrho$ with respect to the norm of $L^{p'}( D )$.
\end{theorem}
\begin{proof}
According to Theorem \ref{th2}, for every $n$ we have the estimate
$$
\|\varrho_{n}\|_{L^{p'}(D)}\le C_n\bigl(\|\eta_n\|_{TV}+\|G_n\|_{L^{p'}(D)}+\|h_n\|_{L^1(D)}\bigr),
$$
where the number $C_n$ depends only on $p$, $D$, $\omega_{\overline{D}}$, $\theta_{\overline{D}}$, $\|b_n\|_{L^p(D)}$
and $\omega_{b_n, D}$. By the definition of an admissible approximation the norms
$$
\|b_n\|_{L^p(D)}, \quad  \|\eta_n\|_{TV}, \quad \|G_n\|_{L^{p'}(D)}, \quad \|h_n\|_{L^1(D)}
$$
are bounded by some number independent of $n$. In addition, for
$\omega_{b_n, D}$ one can take $\widetilde{\omega}_{\overline{D}}$, hence the constant $C_n$ can be made
independent of $n$. Therefore, the sequence $\{\varrho_n\}$ is bounded in $L^{p'}(D)$.
Hence there exists a subsequence $\{\varrho_{n_j}\}$ converging weakly in $L^{p'}( D )$
to some function~$\varrho$. Let us consider a function $u\in W^{2, p}( D )\cap C^1(\overline{ D })$ vanishing on
 $\partial D $ and justify passage to the limit as $n\to\infty$ in the equality
\begin{equation}\label{intc1}
\int_{ D }L_{A_{n_j}, b_{n_j}}u(x)\varrho_{n_j}(x)\,dx
=\int_{ D }L_{G_{n_j}, h_{n_j}}u(x)dx+
\int_{\partial D }\langle A_{n_j}(x)\nabla u(x), \nu(x)\rangle\,\eta_{n_j}\sigma_{d-1}(dx).
\end{equation}
Since $g^{ij}_n\to g^{ij}$ in $L^{p'}( D )$ and $h^i_n\to h^i$ in $L^1( D )$, we have
$$
\lim_{j\to\infty}\int_BL_{G_{n_j}, h_{n_j}}u(x)dx=\int_BL_{G, h}u(x)\,dx.
$$
Observe that
\begin{multline*}
\biggl|\int_{ D }L_{A_{n_j}, b_{n_j}}u(x)\varrho_{n_j}(x)\,dx
-\int_{ D }L_{A, b}u(x)\varrho_{n_j}(x)\,dx\biggr|
\\
\le
\|A_{n_j}-A\|_{L^{\infty}( D )}\|u\|_{W^{p, 2}( D )}\|\varrho_{n_j}\|_{L^{p'}( D )}+
\|b_{n_j}-b\|_{L^p( D )}\|u\|_{C^1(\overline{ D })}\|\varrho_{n_j}\|_{L^{p'}( D )}.
\end{multline*}
Therefore,
$$
\lim_{j\to\infty}\biggl|\int_{ D }L_{A_{n_j}, b_{n_j}}u(x)\varrho_{n_j}(x)\,dx
-\int_{ D }L_{A, b}u(x)\varrho_{n_j}(x)\,dx\biggr|=0.
$$
Thus, for passing to the limit in the left-hand side of (\ref{intc1}) it suffices to observe that
due to weak convergence of the functions $\varrho_{n_j}$ to $\varrho$ we have
$$
\lim_{j\to\infty}\int_{ D }L_{A, b}u(x)\varrho_{n_j}(x)\,dx=
\int_{ D }L_{A, b}u(x)\varrho(x)\,dx.
$$
Let us consider the second term in the right-hand side of (\ref{intc1}). Since the
matrices  $A_n$  converge uniformly to $A$ and, by virtue of weak convergence,
the sequence of measures $\eta_n\sigma_{d-1}$ is bounded in  variation, it suffices to justify the equality
$$
\lim_{j\to\infty}\int_{ \partial D }\langle A(x)\nabla u(x), \nu(x)\rangle\eta_{n_j}\,d\sigma_{d-1}=
\int_{ D }\langle A(x)\nabla u(x), \nu(x)\rangle\, \eta(dx).
$$
By the continuity of $A$ and $\nabla u$ this equality  follows from weak convergence of the measures $\eta_{n_j}\sigma_{d-1}$ to the measure $\eta$.
Thus, letting $j\to\infty$, we obtain the relationship
$$
\int_{ D } L_{A,b}u (x)\varrho(x)\,dx=\int_{\partial D }\langle A(x)\nabla u(x), \nu(x)\rangle\,\eta(dx)
+\int_{ D }L_{G, h}u(x)\,dx.
$$
Therefore, $\varrho$ is a solution to the  Dirichlet problem (\ref{dir}). Since a solution is unique and the reasoning above
can be repeated for every subsequence $\{\varrho_{n_k}\}$ in place of $\{\varrho_n\}$, we conclude that
the whole sequence $\{\varrho_n\}$ converges weakly in $L^{p'}( D )$ to the solution $\varrho$.

Let us prove the last assertion of the corollary.
Suppose that the  measures $\eta_n\sigma_{d-1}$ converge to the measure $\eta$ in variation.

Observe that for every function $u\in W^{p, 2}( D )\cap C^1(\overline{ D })$ vanishing on $\partial  D $ we have
$$
\int_{ D }L_{A, b}u(x)\varrho(x)\,dx
=\int_{ D }L_{G, h}u(x)dx+\int_{\partial D }\langle A(x)\nabla u(x), \nu(x)\rangle\,\eta(dx)
$$
and
$$
\int_{ D }L_{A_n, b_n}u(x)\varrho_n(x)\,dx
=\int_{ D }L_{G_n, h_n}u(x)dx+\int_{\partial D }\langle A_n(x)\nabla u(x), \nu(x)\rangle\,\eta_n(x)\sigma_{d-1}(dx).
$$
Set $\zeta_n=\varrho_n-\varrho$ and
$$
\widetilde{G}_n=(G_n-G)+\varrho_n(A-A_n), \quad \widetilde{h}_n=(h_n-h)+\varrho_n(b-b_n),
$$
$$
\widetilde{\eta}_n=(\eta_n\sigma_{d-1}-\eta)
+\Bigl(\frac{\langle A_n\nu, \nu\rangle}{\langle A\nu, \nu\rangle}-1\Bigr)\eta_n\sigma_{d-1}.
$$
Then we obtain the equality
$$
\int_{ D }L_{A, b}u(x)\zeta_n(x)\,dx
=\int_{ D }L_{\widetilde{G}_n, \widetilde{h}_n}u(x)dx
+\int_{\partial D }\langle A(x)\nabla u(x), \nu(x)\rangle\,\widetilde{\eta}_n(dx).
$$
Here we also use the equality $\nabla u(x)=\langle\nabla u(x), \nu(x)\rangle\nu(x)$ if $x\in \partial D$.
Thus, the function $\zeta_n$ is a solution to the Dirichlet problem for the equation
$$
{\rm div}^2 (\zeta_n A)+{\rm div}(\zeta_n b)
={\rm div}^2\bigl(\widetilde{G}_n\bigr)+{\rm div}\bigl(\widetilde{h_n}\bigr)
$$
with boundary condition $\widetilde{\eta}_n+\widetilde{\kappa}_n\sigma_{d-1}$,
where the function $\widetilde{\kappa}$ is defined by formula~(\ref{dens})
with $\widetilde{G}_n$ in place of~$G$. Note that the measures $\widetilde{\eta}_n$
converge to zero in total variation.
We also have
$$
\|(A-A_n)\varrho_n\|_{L^{p'}(D)}\le \|A-A_n\|_{L^{\infty}(D)}\|\varrho_n\|_{L^{p'}(D)},
$$
$$
\|(b-b_n)\varrho_n\|_{L^1(D)}\le \|b-b_n\|_{L^{p}(D)}\|\varrho_n\|_{L^{p'}(D)}.
$$
Recall that the sequence $\{\varrho_n\}$ is bounded in $L^{p'}(D)$.
Applying the  estimate from Theorem~\ref{th2}, we obtain
\begin{multline*}
\|\zeta_n\|_{L^{p'}(D)}\le C\Bigl(\|\widetilde{\eta_n}\|_{TV}+\|G_n-G\|_{L^{p'}(D)}+\|A_n-A\|_{L^{\infty}(D)}\|\varrho_n\|_{L^{p'}(D)}
\\
+\|h_n-h\|_{L^1(D)}+\|b-b_n\|_{L^{p}(D)}\|\varrho_n\|_{L^{p'}(D)}\Bigr).
\end{multline*}
Therefore, $\{\zeta_n\}$ tends to zero in $L^{p'}(D)$ as $n\to\infty$.
\end{proof}

Thus, the  Dirichlet problem (\ref{dir}) has a unique solution and this solution can be obtained as a weak limit
in $L^{p'}( D )$ of solutions to  Dirichlet problems for equations with smooth coefficients and smooth boundary conditions.

\begin{corollary}\label{col1}
Let $G=0$, $h=0$ and let $\eta$ be a nonnegative measure on $\partial D $.
Then the solution~$\varrho$ to the  Dirichlet problem {\rm (\ref{dir})} is
almost everywhere nonnegative on $ D $.
\end{corollary}
\begin{proof}
By the previous theorem the solution $\varrho$ can be obtained as a weak limit in $L^{p'}( D )$ of solutions $\varrho_n$
to smooth problems, moreover, we can assume that the densities $\eta_n$ mentioned in the theorem  are nonnegative.
Applying the  maximum principle for solutions to the divergence form elliptic equations
$$
{\rm div}(A_n\nabla v- v \widetilde{b}_n)=0,
$$
we conclude that all functions $\varrho_n$ are nonnegative on $ D $. This
implies that the function $\varrho$ is nonnegative.
\end{proof}

The results proved above enable us to construct approximations by smooth solutions
of an arbitrary solution to equation (\ref{eq1}). The proof of the following statement
is based on the approximation procedure elaborated for divergence type equations in \cite{D12}
and employed in~\cite{DK17}. We are most grateful to
Professor Hongjie Dong for his advice to apply his method.

\begin{corollary}\label{col2}
Suppose that under the hypotheses of Theorem {\rm \ref{th1}} the functions $a^{ij}_n$, $b^i_n$, $g^{ij}_n$
and $h^i_n$ are an admissible approximation of the functions $a^{ij}$, $b^i$, $g^{ij}$
and $h^i$.  Assume also that a function $\varrho\in L^{p'}_{loc}(\Omega)$ is a solution to
equation {\rm (\ref{eq1})} and a bounded domain $D$ with smooth boundary is contained with closure in the domain $\Omega$.
Then there exists a sequence of infinitely differentiable solutions $\varrho_n$ to the equations
$$
{\rm div}^2 (v A_n)-{\rm div}(v b_n)={\rm div}^2G_n-{\rm div}\,h_n
$$
on  $D$ such that $\{\varrho_n\}$ converges in the  norm of $L^{p'}(D)$ to the function $\varrho$. Moreover, if
$G=0$, $h=0$ and $\varrho\ge 0$, we can find nonnegative functions $\varrho_n$ on~$D$.
\end{corollary}
\begin{proof}
Let $D_0$ be a bounded domain with smooth boundary such that $\overline{D}\subset D_0$
and $\overline{D_0}\subset \Omega$.
There exists $\delta>0$ such that $D^{4\delta}\subset D_0$ and $D_0^{2\delta}\subset\Omega$.
Let $\zeta\in C_0^{\infty}(D^{2\delta})$ and $\zeta(x)=1$ if $x\in D^{\delta}$.
Set
$$
f={\rm trace}(GD^2\zeta)-\varrho\,{\rm trace}(AD^2\zeta)-\varrho\langle b, \nabla\zeta\rangle
+\langle h, \nabla\zeta\rangle.
$$
Then $f\in L^1(D^{\delta}$. It is known (see \cite[Theorem~1.2.2]{MarcVeron13}) that
there exists a vector field $w$ with $|w|\in L^s(D_0^{\delta})$ for all $s\in [1,d/(d-1))$
such that ${\rm div}\,w=f$ on $D_0$.
Then
$$
{\rm div}^2(\varrho\zeta A)-{\rm div}(\varrho\zeta b)=
{\rm div}^2 G_0-{\rm div}(h_0),
$$
where
$$
G_0=G\zeta, \quad h_0=h\zeta+2G\nabla\zeta-2\varrho A\nabla\zeta-w.
$$
It is clear that $\varrho\zeta$ is a solution to the Dirichlet problem (\ref{dir}) on $D_0$
with zero boundary condition and $G_0, h_0$ in place of $G, h$. Set
$$
G_{0,n}=(g^{ij}_{0,n})=G_n\zeta, \quad
h_{0,n}=(h^{i}_{0,n})=h_n\zeta-G_n\nabla\zeta+(\varrho A\nabla\zeta-w)*\psi_{1/n},
$$
where $\psi_{1/n}(x)=n^{d}\psi(nx)$, $n^{-1}<\delta$ and $\psi\in C_0^{\infty}(\mathbb{R}^d)$
are such that $\psi\ge 0$, $\|\psi\|_{L^1(\mathbb{R}^d)}=1$
and the support of $\psi$ is contained in the open ball $B(0, 1)$.
Note that the functions $a^{ij}_n, b_n^i, g^{ij}_{0,n}, h^i_{0,n}$ provide
an admissible approximation of the functions $a^{ij}$, $b^i$, $g^{ij}$ and $h^i$
on $D_0^{\delta}$. Theorem \ref{th3} gives solutions $\varrho_n$ to the Dirichlet problems
$$
\left\{\begin{array}{lc}
  {\rm div}^2 (v A_n)-{\rm div} ( v b_n)={\rm div}^2 G_{0,n}-{\rm div}\, h_{0,n}, & x\in D_0, \\
  \\
  \varrho_n|_{\partial D_0}=0.
\end{array}\right.
$$
Since the coefficients $a^{ij}_n$, $b^i_n$, $g^{ij}_{0,n}$, $h^i_{0,n}$ are smooth, the
solutions $\varrho_n$ are smooth functions in~$D_0$.
The justification of convergence of $\{\varrho_n\}$ to $\varrho\zeta$ in $L^{p'}(D_0)$
repeats the reasoning used at the end of the proof of Theorem~\ref{th3}.
Note that ${\rm div}(w*\psi_{1/n})=0$ in~$D$. Moreover, on the domain $D$ we have $\varrho\zeta=\varrho$,
$G_{0,n}=G_n$,  ${\rm div}\, h_{0,n}={\rm div}\,h_n$. Thus, on $D$ the function $\varrho_n$
satisfies the equation
$$
{\rm div}^2 (v A_n)-{\rm div} ( vb_n)={\rm div}^2G_n-{\rm div}\,h_n
$$
and the sequence $\{\varrho_n\}$ converges to $\varrho$ in $L^{p'}(D)$.

Let us consider the case where $G=0$, $h=0$ and $\varrho\ge 0$.
According to Theorem \ref{th1}, the function $\varrho$ solves the Dirichlet problem  (\ref{dir}) on
$D$ with some measure $\eta$ on $\partial D$. Theorem \ref{th3} gives solutions $\varrho_n$ to the Dirichlet problems
$$
\left\{\begin{array}{lc}
  {\rm div}^2 (v A_n)-{\rm div} ( b_n v)=0, & x\in D, \\
  \\
  \varrho_n|_{\partial D}=\eta\sigma_{d-1}.
\end{array}\right.
$$
Since the coefficients $a^{ij}_n$, $b^i_n$ are smooth, the
solutions $\varrho_n$ are smooth functions in~$D$. By Corollary \ref{col1} we have $\varrho_n\ge 0$.
The justification of convergence of $\{\varrho_n\}$ to $\varrho$ in $L^{p'}(D)$
repeats the reasoning used at the end of the proof of Theorem~\ref{th3}.
\end{proof}

As an application of the obtained results we shall show how the justification of \cite[Theo\-rem~1.10]{DK17} and \cite[Theprem 2.15]{K25}
can be  reduced to the case of smooth solutions to equations with smooth coefficients.

A continuous function $f$ on $\Omega$ is said to satisfy the Dini mean oscillation condition if there
exists a continuous nondecreasing nonnegative function $\widetilde{\omega}$ on $[0, +\infty)$ such that
$$
\widetilde{\omega}(0)=0, \quad \int_0^1\frac{\widetilde{\omega}(t)}{t}\,dt<\infty,
$$
and for all $r>0$ one has
$$
\sup_{x\in\Omega}\frac{1}{|\Omega(x, r)|}\int_{\Omega(x, r)} |f(y)-f_{\Omega}(x, r)|\,dy\le \widetilde{\omega}(r),
$$
where
$$
f_{\Omega}(x, r)=\frac{1}{|\Omega(x, r)|}\int_{\Omega(x, r)}f(y)\,dy, \quad \Omega(x, r)=\Omega\cap B(x, r).
$$
Recall that for a ball $B$ the symbol $W^{p, 1}_0(B)$ denotes the  closure of $C_0^\infty(B)$ in $W^{p, 1}(B)$;
for $p>d$ this class consists of continuous functions on~$\overline{B}$ vanishing on~$\partial B$ and possessing generalized first order
partial derivatives  from~$L^p(B)$.

\begin{theorem}\label{th4}
Suppose that  conditions {\rm (H1)} and {\rm (H2)} are fulfilled, the functions $g^{ij}$ are continuous, the functions $a^{ij}$
and $g^{ij}$ satisfy the Dini mean oscillation condition and $b^i=h^i=0$.
Let  $\varrho\in L^1_{loc}(\Omega)$ be a solution to equation \eqref{eq1}.
Then the function $\varrho$ possesses a continuous version on~$\Omega$.
\end{theorem}
\begin{proof}
Let $p>d$. We prove that $\varrho\in L^{p'}_{loc}(\Omega)$, where $p'=p/(p-1)$.
Let $B(x_0, 4r)\subset \Omega$ and $\zeta\in C_0^{\infty}(B(x_0, 2r)$, $0\le \zeta\le 1$
and $\zeta(x)=1$ if $x\in B(x_0, r)$. Outside the ball $B(x_0, 2r)$ we extend the function $\zeta$
by zero. Let us take a function $f\in C_{0}^{\infty}(\mathbb{R}^d)$ with support in the ball $B(x_0, r)$.
According  to \cite[Theorem 1.6]{DK17}, there exists a solution $u\in W^{p, 1}(B(x_0, 3r))\cap W^{p, 1}_0(B(x_0, 3r))$ to the equation
${\rm trace} (AD^2u)=f$ in $B(x_0, 3r)$ with zero boundary condition
that is twice continuously differentiable on $B(x_0, 2r)$. Therefore, the function $u\zeta$
can be substituted into the integral equality defining the solution $\varrho$. We have
\begin{multline*}
\int_{\Omega}f(x)\zeta(x)\varrho(x)\,dx
\\
=-\int_{\Omega}\Bigl[2\langle A(x)\nabla u(x), \nabla\zeta(x)\rangle
+u(x){\rm trace}\bigl(A(x)D^2\zeta(x)\bigr)\Bigr]\varrho(x)\,dx
\\
+\int_{\Omega}{\rm trace}\bigl(G(x)D^2(u(x)\zeta(x))\bigr)\,dx.
\end{multline*}
The function $u$ satisfies the estimate $\|u\|_{W^{p, 2}(B(x_0, 3r))}\le C_1\|f\|_{L^{p}(B(x_0, 3r))}$.
By the Sobolev  embedding theorem $u\in C^1(B(x_0, 3r))$, moreover,
$$
\|u\|_{C^1(B(x_0, 3r))}\le C_2\|u\|_{W^{p, 2}(B(x_0, 3r))}
$$
with some constants $C_1$ and $C_2$, depending only on $p$, $B(x_0, 3r)$, the number $\theta$ and the function $\omega$
from (H1) and (H2). Taking into account that the support of $f$ belongs to $B(x_0, r)$, we obtain
$$
\int_{B(x_0, r)}f(x)\varrho(x)\,dx\le C_3\|f\|_{L^p(B(x_0, r))},
$$
where the constant $C_3$ depends only on $B(x_0, 3r)$, the number $\theta$ and the function $\omega$
from (H1) and (H2). By the Riesz theorem the function $\varrho$ belongs to $L^{p'}(B(x_0, r))$.

We now justify the existence of a continuous version of $\varrho$.
Let $B=B(x_0, r)$ be such that  $B(x_0, 3r)\subset\Omega$.
For the proof it suffices to establish the existence  of a continuous version of the solution
on every such ball $B(x_0, r)$. We can find a  sequence of smooth functions
$a^{ij}_n$ and $g^{ij}_n$ such that for every compact set $K\subset\Omega$ the
sequence $\{a^{ij}_n\}$ converges to $a^{ij}$ uniformly on $K$, the sequence $\{g^{ij}_n\}$ converges to $g^{ij}$ uniformly on $K$,
the functions $a^{ij}_n$ on $K$ satisfy conditions (H1) and (H2) with a number $\theta_K$ and a function $\omega_K$ independent of $n$ and
the functions $a^{ij}_n$, $g^{ij}_n$ satisfy on $B(x_0, 2r)$ the Dini mean oscillation condition with a function $\widetilde{\omega}$
independent of $n$. It is clear that the sequence of functions
$a^{ij}_n$, $b^i_n=0$, $g^{ij}_n$ and $h^i_n=0$ provide an admissible approximation of the functions
$a^{ij}$, $b^i=0$, $g^{ij}$ and $h^i=0$. According to Corollary~\ref{col2}, there exists a sequence of smooth solutions $\varrho_n$ to the equations
$$
{\rm div}^2(\varrho_n A_n)-{\rm div}(\varrho_n b_n)={\rm div}^2G_n-{\rm div}\,h_n
$$
on the ball $B(x_0, 2r)$ such that $\{\varrho_n\}$ converges in $L^{p'}(B(x_0, 2r))$ to the function $\varrho$.
Applying to $\varrho_n$ on $B(x_0, r)$  Theorem~1.10 of  \cite{DK17}, which gives an estimate for the modulus of continuity
of~$\varrho_n$ (see \cite[p.~429]{DK17}), we conclude that there exists a continuous function $w_0$ on $[0, +\infty)$ such  that $w_0(0)=0$ and the inequality
$$
|\varrho_n(x)-\varrho_n(y)|\le w_0(|x-y|)
$$
holds for all $x, y\in \overline{B}(x_0, r)$ and all $n$. Moreover, it follows from estimate (2.25) in \cite{DK17}
that $\sup_n\|\varrho_n\|_{L^{\infty}(\overline{B}(x_0, r))}<\infty$.
Thus, the sequence of functions $\varrho_n$ on the ball $\overline{B}(x_0, r)$ satisfies the Ascoli--Arz\'ela theorem, so this sequence
converges uniformly on $\overline{B}(x_0, r)$ to some continuous function $\widetilde{\varrho}$.
It is clear that the functions $\widetilde{\varrho}$ and $\varrho$ coincide almost everywhere on $\widetilde{B}(x_0, r)$.
\end{proof}

\section{Additional remarks}

As we have already mentioned in the introduction,  and shown in the previous theorem, the results of this paper
provide some missing details in the proof of  Theo\-rem~1.10
from \cite{DK17} and in a number of assertions from other papers deduced from it, including our papers \cite{BRS23} and~\cite{BSumn}.
For convenience of further references, we give precise formulations.
First of all, Theorem~5.2 from our paper \cite{BSumn} is true without additional assumptions about the local boundedness of the density $\varrho$
(added in the English translation, but omitted in the Russian original, although implicitly used through a reference
to \cite[Theorem~1.10]{DK17}). Let us give a complete formulation with a comment on the  proof from~\cite{BSumn}.

\begin{theorem}\label{th52}
Suppose that a locally bounded Borel  measure $\mu$ on $\Omega$ {\rm(}possibly, signed{\rm)}
satisfies the equation $L^{*}_{A, b}\mu+c\mu=0$, where $c$ is a Borel function
locally integrable with respect to~$|\mu|$. Let $A$ satisfy the Dini mean oscillation condition,
${\rm det}\,A>0$, $p>d$, and let one of the two conditions be  fulfilled:
{\rm (i)} $b\in L^p_{loc}(\Omega)$, $c\in L^{p/2}_{loc}(\Omega)$ or
{\rm (ii)} $b\in L^p_{loc}(\mu)$, $c\in L^{p/2}_{loc}(\mu)$.
Then the measure $\mu$ has a continuous density~$\varrho$.
\end{theorem}
\begin{proof}
According to \cite[Theorem 3.5]{BSumn}, the measure $\mu$ is given by a density $\varrho$
with respect to Lebesgue  measure and $\varrho\in L^r_{loc}(\Omega)$ for all $r\ge 1$.
Let $B=B(x_0, R)$, where $B(x_0, 2R)\subset\Omega$. Consider the  Dirichlet problem
$$
\Delta w={\rm div}(\varrho b)-c\varrho, \quad w|_{\partial B}=0.
$$
Repeating the reasoning from the proof of Theorem 5.2 in \cite{BSumn}, we conclude that $w\in W^{q, 1}(B)$
with some $q>d$. By the Sobolev  embedding theorem $w$ has a H\"older continuous version. Set $G=w{\rm I}$.
Then ${\rm div}^2(\varrho A)={\rm div}^2G$ and Theorem \ref{th4} applies.
\end{proof}

Next, Theorem 5.1 and Theorem 5.8  from \cite{BSumn}, Theorem 3.5 from \cite{BRS23} and Theorem 4.3 from~\cite{GKim},
 which give sufficient conditions for the validity of the Harnack inequality for nonnega\-tive solutions to double divergence form
elliptic equations are true without the additional assumption about the local boundedness of $\varrho$
(this assumption was added in the English translation of~\cite{BSumn} and omitted in the Russian original,
although also implicitly used through a reference
to \cite[Theorem~1.10]{DK17}).
Here we also give a complete formulation with a comment about the proof.

\begin{theorem}\label{th5158}
Suppose  that a locally bounded  nonnegative Borel measure $\mu$ on $\Omega$
satisfies the  equation $L^{*}_{A, b}\mu+c\mu=0$, where $c$ is a Borel function
locally integrable with respect to~$\mu$. Let $A$ satisfy the Dini mean oscillation condition,
${\rm det}\,A>0$, $p>d$, and let $b\in L^p_{loc}(\Omega)$, $c\in L^{p/2}_{loc}(\Omega)$.
Then $\mu$ has a  continuous density~$\varrho$ and for every ball $B(x_0, R)$ with $B(x_0, 4R)\subset\Omega$
there is a  number $C>0$ such that
$$
\sup_{x\in B(x_0, R)}\varrho(x)\le C\inf_{x\in B(x_0, R)}\varrho(x),
$$
where $C$ depends only on $d$, $R$, $\sup_{B(x_0, 4R)}(\|A(x)\|+\|A(x)^{-1}\|)$, the function $\widetilde{\omega}$
from the Dini mean oscillation condition for the matrix $A$ and on the quantities $\|b\|_{L^p(B(x_0, 4R))}$ and $\|c\|_{L^{p/2}(B(x_0, 4R))}$.
\end{theorem}
\begin{proof}
Applying Theorem \ref{th52}, we conclude that  $\mu$ has a continuous density~$\varrho$. Now without any changes the reasoning
from~\cite{GKim} applies.
\end{proof}

It follows that $\varrho$ is positive if $\Omega$ is connected and $\mu$ is not zero.

It would be interesting to obtain analogous results for double divergence form parabolic equations.

{\bf Acknowledgements}
We thank Professor Hongjie Dong for helpful suggestions.
This research was supported by the Russian Science Foun\-da\-tion Grant 25-11-00007 at Lo\-mo\-no\-sov
Moscow State University.

\end{document}